\documentclass[11pt]{amsart}%
\usepackage{hyperref}
\usepackage{amsmath}%
\usepackage{amsfonts}%
\usepackage{amssymb}%
\usepackage{amsthm}
\usepackage{mathrsfs}
\usepackage{comment}
\usepackage{todonotes}
\usepackage{bbm}
\usepackage{amscd}
\usepackage{youngtab}
\usepackage{tikz}
\usepackage{bm}
\usepackage{enumitem}
\usetikzlibrary{arrows}
\usetikzlibrary{matrix}

\newcommand{\Q}{\mathbb{Q}}
\newcommand{\C}{\mathbb{C}}
\newcommand{\Z}{\mathbb{Z}}

\newcommand{\N}{\mathbb{N}}
\newcommand{\mf}{\mathfrak}

\DeclareMathOperator{\wt}{wt}
\DeclareMathOperator{\Cov}{Cov}
\DeclareMathOperator{\spn}{span}
\newcommand{\bs}{\boldsymbol}

\renewcommand{\tilde}{\widetilde}
\renewcommand{\S}{\mathfrak{S}}
\renewcommand{\SS}{\tilde{\mathfrak{S}}}

\newtheorem{theorem}{Theorem}[section]
\newtheorem{def-prop}[theorem]{Definition-Proposition}
\newtheorem{prop}[theorem]{Proposition}

\newtheorem{lemma}[theorem]{Lemma}

\theoremstyle{definition}
\newtheorem{ex}[theorem]{Example}
\newtheorem{definition}[theorem]{Definition}

\theoremstyle{remark}
\newtheorem*{remark}{Remark}

\begin{document}

\title[Weighted enumeration of Bruhat chains]{Padded Schubert polynomials and weighted enumeration of Bruhat chains}
\author{Christian Gaetz}
\thanks{C.G. was supported by a National Science Foundation Graduate Research Fellowship under Grant No. 1122374}
\author{Yibo Gao}
\address{Department of Mathematics, Massachusetts Institute of Technology, Cambridge, MA.}
\email{\href{mailto:gaetz@mit.edu}{gaetz@mit.edu}} 
\email{\href{mailto:gaoyibo@mit.edu}{gaoyibo@mit.edu}}
\date{\today}

\begin{abstract}
We use the recently-introduced \emph{padded Schubert polynomials} to prove a common generalization of the fact that the weighted number of maximal chains in the strong Bruhat order on the symmetric group is ${n \choose 2}!$ for both the \emph{code weights} and the \emph{Chevalley weights}, generalizing a result of Stembridge \cite{Stembridge}.  We also define weights which give a one-parameter family of strong order analogues of Macdonald's well-known reduced word identity for Schubert polynomials \cite{Macdonald}.
\end{abstract}

\maketitle

\section{Introduction} \label{sec:intro}

Let $\bs{S}_n$ denote the strong Bruhat order on the symmetric group $S_n$ (see Section \ref{sec:background} for background and definitions).  Given a function $\wt:\Cov(\bs{S}_n) \to R$ from the set of covering relations of $\bs{S}_n$ to a ring $R$, and a saturated chain $c=(u_1 \lessdot u_2 \lessdot \cdots \lessdot u_k)$, we define the weight of $c$ mutliplicatively:
\[
\wt(c):=\prod_{i=1}^{k-1} \wt(u_i \lessdot u_{i+1}).
\]
For $v \leq w$ in $\bs{S}_n$ we let
\[
m_{\wt}(v,w):=\sum_{c:v \to w} \wt(c)
\]
denote the total weighted number of chains over all saturated chains $c$ from $v$ to $w$.

In this paper, we study several classes of weights which generalize the recently introduced \emph{code weights} \cite{duality-paper} and the more classical \emph{Chevalley weights}, studied, for example, by Postnikov and Stanley \cite{Postnikov-Stanley} and by Stembridge \cite{Stembridge}.  Some building blocks for these new weights are given in Definition \ref{def:matrix-weights}.

\begin{definition} \label{def:matrix-weights}
For $v \lessdot w=vt_{ij}$ a covering relation in $\bs{S}_n$ with $i<j$, let $a_{v \lessdot w}, b_{v \lessdot w}, c_{v \lessdot w}$, and $d_{v \lessdot w}$ denote the number of dots in the regions $A,B,C$, and $D$ respectively in Figure \ref{fig:matrix-weights}.  That is,
\begin{align*}
    a_{v \lessdot w} &= \# \{ k < i \: | \: v_i < v_k < v_j \} \\
    b_{v \lessdot w} &= \# \{ i < k < j \: | \: v_k > v_j \} \\
    c_{v \lessdot w} &= \# \{ k > j \: | \: v_i < v_k < v_j \} \\
    d_{v \lessdot w} &= \# \{ i < k < j \: | \: v_k < v_i \}.
\end{align*}
Note that we always have $b_{v \lessdot w} + d_{v \lessdot w}=j-i-1$ and $a_{v \lessdot w} + c_{v \lessdot w}=v_j-v_i-1$.
\end{definition}

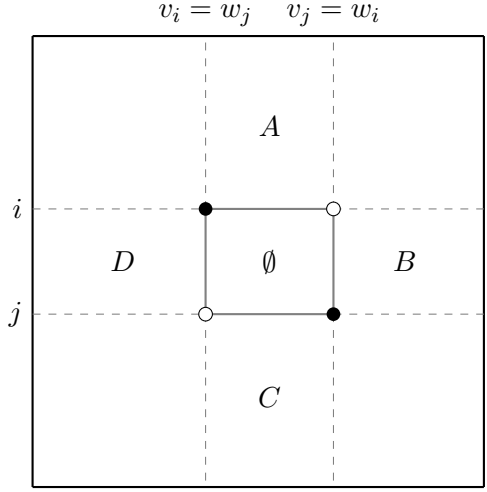
\begin{figure}[h] 
    \centering
    \begin{tikzpicture}
    \draw[black, thick] (-3,3)--(3,3);
    \draw[black, thick] (-3,-3)--(3,-3);
    \draw[black, thick] (-3,3)--(-3,-3);
    \draw[black, thick] (3,3)--(3,-3);
    
    \draw[gray, thick] (-0.7,0.7)--(1,0.7);
    \draw[gray, thick] (-0.7,-0.7)--(1,-0.7);
    \draw[gray, thick] (-0.7,0.7)--(-0.7,-0.7);
    \draw[gray, thick] (1,0.7)--(1,-0.7);
    
    \draw[gray, dashed] (-3,-0.7)--(-0.7,-0.7);
    \draw[gray, dashed] (-3,0.7)--(-0.7,0.7);
    \draw[gray, dashed] (-0.7,0.7)--(-0.7,3);
    \draw[gray, dashed] (-0.7,-0.7)--(-0.7,-3);
    \draw[gray, dashed] (1,-0.7)--(1,-3);
    \draw[gray, dashed] (1,-0.7)--(3,-0.7);
    \draw[gray, dashed] (1,0.7)--(3,0.7);
    \draw[gray, dashed] (1,0.7)--(1,3);
    
    \node[circle,fill=white,draw,scale=0.5](aji) at (-0.7,-0.7) {};
    \node[circle,fill=white,draw,scale=0.5](aij) at (1,0.7) {};
    \node[circle,fill=black,scale=0.5](aii) at (-0.7,0.7) {};
    \node[circle,fill=black,scale=0.5](ajj) at (1,-0.7) {};
    
    \node(i) at (-3,0.7) [left]{$i$};
    \node(j) at (-3,-0.7) [left]{$j$};
    \node(wi) at (-0.7,3) [above]{$v_i=w_j$};
    \node(wj) at (1,3) [above]{$v_j=w_i$};
    
    \node(A) at (.15,1.8) {$A$};
    \node(B) at (1.95,0) {$B$};
    \node(C) at (.15,-1.8) {$C$};
    \node(D) at (-1.8,0) {$D$};
    \node(middle) at (.15,0) {$\emptyset$};
    \end{tikzpicture}
    \caption{For $v \lessdot w$ a covering relation in the strong order, the permutation matrices for $v$ and $w$ agree, except that the black dots in $v$ are replaced with the white dots in $w$.  No dots may occupy the central region; the numbers of dots in the four labeled regions $A,B,C$ and $D$ are used in Definition \ref{def:matrix-weights}.}
    \label{fig:matrix-weights}
\end{figure}

The \emph{Chevalley weights} $\wt_{Chev}(v \lessdot w): \Cov(\bs{S}_n) \to \Z[\alpha_1,...,\alpha_{n-1}]$ assign weight $\alpha_i+ \cdots + \alpha_{j-1}$ to the covering relation $v \lessdot w=vt_{ij}$, where $t_{ij}=(i \: j)$ is a transposition.  It was shown by Stembridge \cite{Stembridge} that:
\begin{equation} \label{eq:chev-weight-maximal-vars}
m_{Chev}(e,w_0)(\alpha_1,...,\alpha_{n-1})={n \choose 2}!\cdot \prod_{1 \leq k < \ell \leq n-1} \frac{\alpha_k + \cdots + \alpha_{\ell-1}}{\ell-k}
\end{equation}
where $w_0=n (n-1) \cdots 2 1$ denotes the longest permutation.  Specializing all $\alpha_i=1$ recovers a classical fact, due essentially to Chevalley:
\begin{equation}
\label{eq:chev-weight-maximal}
m_{Chev}(e,w_0)(1,...,1)={n \choose 2}!.
\end{equation}

Recently, a new set of weights, the \emph{code weights} $\wt_{code}: \Cov(\bs{S}_n) \to \N$ were defined in the course of proving the Sperner property for the weak Bruhat order \cite{sperner-paper}.  In the notation of Definition \ref{def:matrix-weights}, the code weights are defined by $\wt_{code}(v \lessdot w)=1+2b_{v\lessdot w}$.  In \cite{duality-paper}, it was shown that 
\begin{equation} \label{eq:strong-macdonald}
m_{code}(w,w_0)=\left({n \choose 2} - \ell(w)\right)! \cdot \mf{S}_w(1,...,1)
\end{equation} 
where $\mf{S}_w$ is the \emph{Schubert polynomial} (see Section \ref{sec:background}), providing a strong Bruhat order analogue of Macdonald's well known identity for $\mf{S}_w(1,...,1)$ as a weighted enumeration of chains in the \emph{weak} Bruhat order (see Proposition \ref{prop:original-macdonald}).  Letting $w=e$ in (\ref{eq:strong-macdonald}) gives:
\begin{equation} \label{eq:code-weight-maximal}
    m_{code}(e,w_0)={n \choose 2}!.
\end{equation}
One motivation of this work is to understand and generalize the coincidence between (\ref{eq:chev-weight-maximal}) and (\ref{eq:code-weight-maximal}); this is done in Theorem \ref{thm:any-pair}.  This theorem, together with the dual weak and strong order identities of \cite{Hamaker} and \cite{duality-paper}, also gives the first explanation of why the weighted path counts $e \to w_0$ for Macdonald's weak order weights and Chevalley's strong order weights should agree.

\begin{theorem} \label{thm:any-pair}
Let $f: \Cov(\bs{S}_n) \to \Z[z_A,z_B,z_C,z_D]$ be the weight function defined by
$$f(v \lessdot w):=1+ a_{v \lessdot w}z_A+b_{v \lessdot w}z_B+c_{v \lessdot w}z_C+d_{v \lessdot w}z_D.$$
Let $\wt: \Cov(\bs{S}_n) \to \Z[z]$ be any weight function obtained from $f$ by specializing the variables so that $\{z_A,z_B,z_C,z_D\}=\{0,0,z,2-z\}$ as multisets, then:
\[
m_{\wt}(e,w_0)={n \choose 2}!.
\]
In particular, $m_{\wt}(e,w_0)$ does not depend on $z$.
\end{theorem}

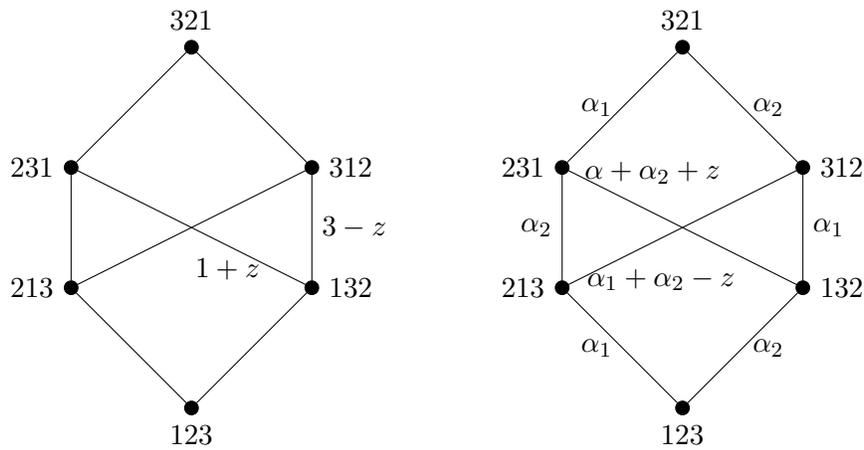
\begin{figure}[h]
    \centering
    \begin{tikzpicture} [scale=1.6]
\node[draw,shape=circle,fill=black,scale=0.5] (a0)[label=below:{$123$}] at (0,0) {};
\node[draw,shape=circle,fill=black,scale=0.5](b0)[label=left:{$213$}] at (-1,1) {};
\node[draw,shape=circle,fill=black,scale=0.5](b1)[label=right:{$132$}] at (1,1) {};
\node[draw,shape=circle,fill=black,scale=0.5](c0)[label=left:{$231$}] at (-1,2) {};
\node[draw,shape=circle,fill=black,scale=0.5](c1)[label=right:{$312$}] at (1,2) {};
\node[draw,shape=circle,fill=black,scale=0.5](d0)[label={$321$}] at (0,3) {};

\node(l)[label={$1+z$}] at (0.3,0.9) {};

\draw (a0) -- node[left] {} (b0);
\draw (a0) -- node[right] {} (b1);
\draw (b0) -- node[left] {} (c0);
\draw (b1) -- node[right] {$3-z$} (c1);
\draw (b0) -- node[below] {} (c1);
\draw (b1) -- node[below] {} (c0);
\draw (c0) -- node[left] {} (d0);
\draw (c1) -- node[right] {} (d0);
\end{tikzpicture} \hspace{0.4in}
\begin{tikzpicture} [scale=1.6]
\node[draw,shape=circle,fill=black,scale=0.5] (a0)[label=below:{$123$}] at (0,0) {};
\node[draw,shape=circle,fill=black,scale=0.5](b0)[label=left:{$213$}] at (-1,1) {};
\node[draw,shape=circle,fill=black,scale=0.5](b1)[label=right:{$132$}] at (1,1) {};
\node[draw,shape=circle,fill=black,scale=0.5](c0)[label=left:{$231$}] at (-1,2) {};
\node[draw,shape=circle,fill=black,scale=0.5](c1)[label=right:{$312$}] at (1,2) {};
\node[draw,shape=circle,fill=black,scale=0.5](d0)[label={$321$}] at (0,3) {};

\node(l1)[label={$\alpha_1+\alpha_2-z$}] at (-0.18,0.81) {};
\node(l2)[label={$\alpha+\alpha_2+z$}] at (-0.25,1.7) {};

\draw (a0) -- node[left] {$\alpha_1$} (b0);
\draw (a0) -- node[right] {$\alpha_2$} (b1);
\draw (b0) -- node[left] {$\alpha_2$} (c0);
\draw (b1) -- node[right] {$\alpha_1$} (c1);
\draw (b0) -- node {} (c1);
\draw (b1) -- node {} (c0);
\draw (c0) -- node[left] {$\alpha_1$} (d0);
\draw (c1) -- node[right] {$\alpha_2$} (d0);
\end{tikzpicture}
    \caption{The weights considered in Theorem \ref{thm:generalized-macdonald} (left) and in Theorem \ref{thm:generalized-chevalley} (right) for $\bs{S}_3$.  Unlabelled edges have weight 1. }
    \label{fig:S3-weights}
\end{figure}

Theorem \ref{thm:generalized-chevalley} provides a common generalization of (\ref{eq:chev-weight-maximal-vars}) and (\ref{eq:code-weight-maximal}); see Example \ref{ex:specializations}.

\begin{theorem} \label{thm:generalized-chevalley}
Let $\wt: \Cov(\bs{S}_n) \to \Z[\alpha_1,...,\alpha_{n-1},z]$ be defined by 
\[
\wt(v \lessdot vt_{ij})=\alpha_i + \alpha_{i+1} + \cdots + \alpha_{j-1} + (b_{v \lessdot vt_{ij}}-d_{v \lessdot vt_{ij}})z.
\]
Then 
\[
m_{\wt}(e,w_0)={n \choose 2}! \cdot \prod_{k < \ell} \frac{\alpha_k + \cdots + \alpha_{\ell-1}}{\ell-k}.
\]
In particular, $m_{\wt}(e,w_0)$ does not depend on $z$.
\end{theorem}

Theorem \ref{thm:generalized-macdonald} extends (\ref{eq:strong-macdonald}) to a one-parameter family of strong Bruhat analogues of Macdonald's identity.

\begin{theorem} \label{thm:generalized-macdonald}
Let $\wt: \Cov(\bs{S}_n) \to \Z[z]$ be defined by 
\[
\wt(v \lessdot w)=1+b_{v \lessdot w}(2-z) +c_{v \lessdot w}z.
\]
Then for any $w \in S_n$ we have 
\[
m_{\wt}(w,w_0)=\left({n\choose 2}-\ell(w)\right)!\cdot \mf{S}_w(1,...,1).
\]
In particular, $m_{\wt}(w,w_0)$ does not depend on $z$.
\end{theorem}

\begin{ex} \label{ex:specializations}
Various specializations of the above Theorems give previously known results:
\begin{enumerate}
    \item Letting $z_B=2$ and $z_A=z_C=z_D=0$ in Theorem \ref{thm:any-pair} recovers (\ref{eq:code-weight-maximal}), while letting $z_B=z_D=1$ and $z_A=z_C=0$ recovers (\ref{eq:chev-weight-maximal}).
    
    \item Letting all $z=\alpha_1=\cdots =\alpha_{n-1}=1$ in Theorem \ref{thm:generalized-chevalley} the weight becomes:
    \begin{align*}
    \wt(v \lessdot w=vt_{ij})&=(j-i)+(b_{v\lessdot w} - d_{v \lessdot w})\\
    &=(b_{v\lessdot w} + d_{v \lessdot w} + 1)+ (b_{v\lessdot w} - d_{v \lessdot w})\\
    &=1+2b_{v \lessdot w}.
    \end{align*}
    This recovers the identity (\ref{eq:code-weight-maximal}) for the code weights.
    
    \item Letting $z=0$ in Theorem \ref{thm:generalized-chevalley} recovers Stembridge's identity (\ref{eq:chev-weight-maximal-vars}) for the Chevalley weights.
    
    \item Letting $z=0$ in Theorem \ref{thm:generalized-macdonald} recovers the strong order Macdonald identity (\ref{eq:strong-macdonald}).
\end{enumerate}
\end{ex}

Section \ref{sec:background} covers background and definitions.  Theorem \ref{thm:generalized-chevalley} is proven in Section \ref{sec:generalized-chevalley}. Section \ref{sec:generalized-macdonald} proves Theorem \ref{thm:generalized-macdonald} and deduces Theorem \ref{thm:any-pair} from Theorems \ref{thm:generalized-chevalley} and \ref{thm:generalized-macdonald}.

\section{Background and definitions} \label{sec:background}

\subsection{Bruhat order}

Let $s_1,...,s_{n-1}$ denote the adjacent transpositions in the symmetric group $S_n$.  For any permutation $w \in S_n$, its \emph{length} $\ell(w)$ is the minimal number of simple transpositions needed to write $w=s_{i_1} \cdots s_{i_{\ell}}$ as a product.  

The (strong) Bruhat order $\bs{S}_n=(S_n, \leq)$ is defined by its covering relations: $v \lessdot w$ whenever $w=vt_{ij}$ for some $i,j$ and $\ell(w)=\ell(v)+1$; this order relation encodes the containment of Schubert varieties in the flag variety.  The Bruhat order has unique minimal element the identity permutation $e$, and unique maximal element $w_0=n (n-1) ... 2 1$ of length ${n \choose 2}$, called the \emph{longest element}.  The Hasse diagram of $\bs{S}_3$ is shown in Figure \ref{fig:S3-weights}.  

It is well known that the maps $v \mapsto w_0v$ and $v \mapsto vw_0$ are antiautomorphisms of the Bruhat order $\bs{S}_n$ and that $v \mapsto v^{-1}$ is an automorphism \cite{bjorner-brenti}; Proposition \ref{prop:symmetries} determines the effect of these maps on the quantities $a,b,c,$ and $d$ from Definition \ref{def:matrix-weights}.

\begin{prop} \label{prop:symmetries}
Let $v \lessdot w$ be a covering relation in $\bs{S}_n$.
\begin{enumerate}
    \item $a_{v \lessdot w}=d_{v^{-1} \lessdot w^{-1}}$ and $b_{v \lessdot w}=c_{v^{-1} \lessdot w^{-1}}$,
    \item $b_{v \lessdot w}=d_{w_0w \lessdot w_0v}$, and 
    \item $a_{v \lessdot w}=c_{ww_0 \lessdot vw_0}$.
\end{enumerate}
\end{prop}
\begin{proof}
These are clear from Figure \ref{fig:matrix-weights} after observing that inversion corresponds to reflecting the permutation matrix across the main (top-left to bottom-right) diagonal, that left multiplication by $w_0$ corresponds to reflecting across the vertical axis, and that right multiplication by $w_0$ corresponds to reflecting across the horizontal axis.
\end{proof}

\subsection{Schubert polynomials and padded Schubert polynomials}

For $w \in S_n$ the \emph{Schubert polynomials} $\mf{S}_w(x_1,...,x_n)$, introduced by Lascoux and Sch\"{u}tzenberger \cite{Lascoux-Schutzenberger}, represent the classes of Schubert varieties in the cohomology $H^*(G/B)$ of the flag variety.  They can be defined recursively as follows:
\begin{itemize}
    \item $\mf{S}_{w_0}(x_1,...,x_n)=x_1^{n-1}x_2^{n-1} \cdots x_{n-2}^2 x_1=\bs{x}^{\rho}$, where $\rho=(n-1,n-2,...,1)$ denotes the staircase composition, and 
    \item $\mf{S}_{ws_i}=N_i \cdot \mf{S}_w$ when $\ell(ws_i)<\ell(w)$.
\end{itemize}
Here $N_i$ denotes the $i$-th \emph{Newton divided difference operator}:
\[
N_i \cdot g(x_1,...,x_n) := \frac{g(x_1,...,x_n) - g(x_1,...,x_{i+1}, x_i,...,x_n)}{x_i-x_{i+1}}.
\]
The Schubert polynomials $\{\mf{S}_w\}_{w \in S_n}$ form a basis for the vector space $V_n=\spn_{\Q} \{\bs{x}^{\gamma} \: | \: \gamma \leq \rho\}$, where here $\leq$ denotes component-wise comparison.  Let $\tilde{V}_n=\spn_{\Q}\{\bs{x}^{\gamma}\bs{y}^{\rho-\gamma} \: | \: \gamma \leq \rho\}$, then the \emph{padded Schubert polynomials} $\tilde{\mf{S}}_w$, introduced in \cite{duality-paper}, are defined as the images of the $\mf{S}_w$ under the natural map $\bs{x}^{\gamma} \mapsto \bs{x}^{\gamma}\bs{y}^{\rho-\gamma}$ from $V_n \to \tilde{V}_n$.  Define a differential operator $\Delta: \tilde{V}_n \to \tilde{V}_n$ by
\[
\Delta = \sum_{i=1}^n x_i \frac{\partial}{\partial y_i}.
\]
\begin{prop}[\cite{duality-paper}] \label{prop:padded}
For any $w \in S_n$ we have:
$$\Delta\SS_w=\sum_{u:\ w\lessdot u}(1+2b_{w\lessdot u})\SS_{u}.$$
\end{prop}

We will also need Macdonald's well-known reduced word identity for the number of monomials in $\mf{S}_w$:

\begin{prop}[\cite{Macdonald}]
\label{prop:original-macdonald}
\[
\mf{S}_w(1,...,1)=\frac{1}{\ell(w)!} \sum_{s_{a_1}\cdots s_{a_{\ell(w)}} = w} a_1 \cdots a_{\ell(w)}.
\]
\end{prop}

\section{Proof of Theorem \ref{thm:generalized-chevalley}}
\label{sec:generalized-chevalley}
We will modify a proof idea for (\ref{eq:chev-weight-maximal-vars}) due to Stanley \cite{Stanley-Bruhat}. Let us define some linear operators on the \emph{coinvariant algebra} $\C[x_1,\ldots,x_n]/I$, where $I$ is the ideal generated by all symmetric polynomials in $x_1,\ldots,x_n$ with vanishing constant terms. The core of the argument comes from interpreting the operator $\Delta$ with respect to two different bases of $\C[x_1,\ldots, x_n]/I$: one is $\{\S_w|w\in S_n\}$ and the other one is $\{\bs{x}^{\gamma}|\gamma\leq\rho\}$.

Recall that we have defined $\Delta$ on $\tilde V_n$. We can define it naturally on $V_n$ since $V_n\rightarrow\tilde V_n$ is an isomorphism. Namely, it can be seen from the definition that $\Delta x^{\gamma}=\big(\sum_{i=1}^{n}(n-i-\gamma_i)x_i\big)x^{\gamma}$ for $\gamma\leq\rho$ (this operator $\Delta$ appears implicitly in \cite{Hamaker}). Moreover, we can extend this definition of $\Delta$ to $\C[x_1,\ldots,x_n]$ by the same formula. We claim that such definition is in fact well-defined on the quotient $\C[x_1,\ldots,x_n]/I$. This is formulated in the following technical lemma, which is necessary for the correctness of the main proof but is not related to the key idea of the proof.
\begin{lemma}\label{lem:Delta-well-defined}
The linear operator $\Delta:x^{\gamma}\mapsto\big(\sum_{i=1}^{n}(n-i-\gamma_i)x_i\big)x^{\gamma}$ is well-defined on $\C[x_1,\ldots,x_n]/I$ and coincides with $\sum_{i=1}^nx_i\frac{\partial}{\partial y_i}$ on $\tilde V_n$.
\end{lemma}
\begin{proof}
We need to check that if $f\in I$, then $\Delta f\in I$. For convenience, we will first pad every monomial $x^{\gamma}$ to $x^{\gamma}y^{\rho-\gamma}$, allowing negative exponents on $y$-variables, so that we can use $\Delta=\sum_i\frac{\partial}{\partial y_i}x_i$, and then specialize $y_i$'s to 1. This is compatible with the definition as in the statement of the lemma. This means $\Delta(fg)=f\Delta(g)+g\Delta(f)$. As a result, it suffices to check if $f$ is a generator of $I$, then $\Delta f\in I$.

Let us pick the power sum symmetric functions $f=x_1^k+\cdots+x_n^k$ as generators, for $k\geq1$. After padding, we get $\sum_j(\frac{x_j}{y_j})^ky^\rho.$ Then
\begin{align*}
\Delta\left(\sum_{j=1}^n\left(\frac{x_j}{y_j}\right)^k y^\rho\right)=&\left(\sum_{i=1}^n\frac{\partial}{\partial y_i}x_i\right)\left(\sum_{j=1}^n\left(\frac{x_j}{y_j}\right)^k y^\rho\right)\\
=&\sum_{i,j=1}^nx_i\frac{x_j^k}{y_j^k}\left(\frac{\partial}{\partial y_i}y^\rho\right)+\sum_{i=1}^n
x_iy^\rho\left(\frac{\partial}{\partial y_i}\frac{x_i^k}{y_i^k}\right)\\
=&\left(\sum_{i=1}^nx_i\frac{\partial}{\partial y_i}y^\rho\right)\left(\sum_{j=1}^n\frac{x_j^k}{y_j^k}\right)-(k+1)y^\rho\sum_{i=1}^n\frac{x_i^{k+1}}{y_i^{k+1}}.
\end{align*}
It is clear that both terms belong to $I$ after specializing $y_i$'s to 1. So we are done.
\end{proof}

We will also make use of the following classical identity, called Monk's rule (or Monk's formula).

\begin{prop}[Monk's rule, see e.g. \cite{Manivel}]
\label{prop:monks-rule}
For any $1 \leq m < n$, the identity 
\[
(x_1+\cdots + x_m)\mf{S}_w = \sum_{\substack{j \leq m < k \\ \ell(wt_{jk})=\ell(w)+1}} \mf{S}_{wt_{jk}}
\]
holds in $\C[x_1,\ldots,x_n]/I$.
\end{prop}

Now let $\alpha_1,\ldots,\alpha_{n-1}$ be as in Theorem~\ref{thm:generalized-chevalley} and define a linear operator $M$ as multiplication by 
\begin{align*}
&\alpha_1x_1+\alpha_2(x_1+x_2)+\alpha_3(x_1+x_2+x_3)+\cdots+\alpha_{n-1}(x_1+\cdots+x_{n-1})\\
=&\beta_1x_1+\beta_2x_2+\cdots+\beta_{n-1}x_{n-1}
\end{align*}
where $\beta_i=\alpha_i+\cdots+\alpha_{n-1}$. By Monk's rule,
$$M\S_w=\sum_{w\lessdot wt_{ij}}(\alpha_i+\cdots+\alpha_{j-1})\S_{wt_{ij}}.$$
Note that Monk's rule only holds modulo the ideal $I$, and not as an identity of polynomials.  Define another linear operator $R$ by
$$R\S_w=\sum_{u:\ w\lessdot u}(b_{w\lessdot u}-d_{w\lessdot u})\S_u.$$
Write $M_1=M$ and define $M_{k+1}=[M_k,R]:=M_kR-RM_k$ for $k\geq1$. Here, $[,]$ is the standard Lie bracket. 

\begin{lemma}\label{lem:MR-RM}
The operator $M_k$ is the same as multiplication by the element $(k-1)!(\beta_1x_1^k+\beta_2x_2^k+\cdots+\beta_{n-1}x_{n-1}^k).$
\end{lemma}
\begin{proof}
Let's analyze $R$ a bit more. We have
\begin{align*}
R\S_w=&\sum_{u:\ w\lessdot u}(b_{w\lessdot u}-d_{w\lessdot u})\S_u\\
=&\sum_{u:\ w\lessdot u}(1+2b_{w\lessdot u})\S_u-\sum_{u:\ w\lessdot u}(1+b_{w\lessdot u}+d_{w\lessdot u})\S_u\\
=&\sum_{u:\ w\lessdot u}(1+2b_{w\lessdot u})\S_u-\sum_{w\lessdot wt_{ij}}(j-i)\S_{wt_{ij}}\\
=&\Delta\S_w-\big((n-1)x_1+(n-2)x_2+\cdots+x_{n-1}\big)\S_w
\end{align*}
where the last equality follows from Proposition~\ref{prop:padded} and Monk's rule (as a special case of $M$ by assigning $\alpha_1=\cdots=\alpha_{n-1}=1$). 

We use induction on $k$. Since multiplications by polynomials commute with each other, we have $M_kR-RM_k=M_k\Delta-\Delta M_k$. Let's compute what it does on monomials $\bf{x}^\gamma$:
\begin{align*}
x^{\gamma}&=M_k\left(\sum_{i=1}^{n-1}(n-i-\gamma_i)x_i\right)x^{\gamma}\\ &=(k-1)!\left(\sum_{i,j=1}^{n-1}(n-i-\gamma_i)\beta_jx_ix_j^k\right)x^{\gamma}    
\end{align*}
while on the other hand,
\begin{align*}
\Delta M_kx^{\gamma}=&(k-1)!\Delta \sum_j \beta_j x_1^{\gamma_1}\cdots x_{j-1}^{\gamma_{j-1}}x_j^{\gamma_j+k}x_{j+1}^{\gamma_{j+1}}\cdots x_{n-1}^{\gamma_{n-1}}\\
=&(k-1)!\left(\sum_{i\neq j}(n-i-\gamma_i)\beta_jx_ix_j^k\right)x^{\gamma}\\
&+(k-1)!\left(\sum_i(n-i-\gamma_i-k)\beta_ix_i^{k+1}\right)x^{\gamma}.
\end{align*}
Here, the calculation of $\Delta Mx^{\gamma}$ uses the fact that $\Delta$ is defined on all of $\C[x_1,\ldots,x_n]/I$ (Lemma~\ref{lem:Delta-well-defined}), since the coefficient of $x_j$ may exceed $n-j$. As a result, we see that $(M_k\Delta-\Delta M_k)x^{\gamma}=k!(\sum_{i=1}^{n-1}\beta_ix_i^{k+1})x^{\gamma}$. So the induction step goes through.
\end{proof}
\begin{remark}
In fact, the operator $R$ can be more elegantly written as $$Rf=y^{\rho}\cdot\Delta(f/y^{\rho}),$$ when $f\in\tilde V_n$ is already padded.
\end{remark}
\begin{lemma}\label{lem:MS=0}
View $M_k=(k-1)!\sum_{i=1}^{n-1}\beta_ix_i^k$ as polynomials.  If $\sum kp_k={n\choose 2}$ and $p_1<{n\choose 2}$, then $\prod_{k}M_k^{p_k}$ lies in the ideal $I$ .
\end{lemma}
\begin{proof}
Write $M_k=(k-1)!\sum_{i=1}^{n}\beta_ix_i^k$ with $\beta_n=0$. As $\prod_{k}M_k^{p_k}$ is homogeneous of degree ${n\choose 2}$, we can write it as $f\S_{w_0}$ modulo $I$, where $f$ depends only on $\beta_i$'s. In fact, we can obtain $f\S_{w_0}$ by first multiplying out $\prod_{k}M_k^{p_k}$ and then performing subtraction with respect to the homogeneous part of degree ${n\choose 2}$ in $\C[x_1,\ldots,x_n]/I$. This shows that $f$ is a polynomial of degree at most ${n\choose 2}-2\ell+\ell={n\choose 2}-\ell$. 

On the other hand, if $\beta_i=\beta_{i+1}$, then $\prod_{k\geq1}M_k^{p_k}$ is symmetric in $x_i$ and $x_{i+1}$. Consequently, $N_i(\prod_{k}M_k^{p_k})=0$, where $N_i$ is the $i$-th divided difference operator introduced in Section~\ref{sec:background}. But $0=N_i(f\S_{w_0})=f(N_i\S_{w_0})=f\S_{w_0s_i}$. As $\S_{w_0s_i}\neq0$, we deduce that $(\beta_i-\beta_{i+1})|f$. By symmetry, $(\beta_i-\beta_j)|f$ for all $1\leq i<j\leq n$. As $\deg_{\beta} f=\sum p_k<{n\choose 2}$, we conclude that $f=0$.
\end{proof}
\begin{lemma}\label{lem:MZR=M}
With $M,R$ as above, $(M+zR)^{n\choose 2}\cdot1=M^{n\choose 2}\cdot1.$
\end{lemma}
\begin{proof}
Notice that $R\cdot1=R\cdot \S_{e}=0$ as $b_{w\lessdot s_i}=d_{w\lessdot s_i}=0$. The rest is a simple consequence of Lemma~\ref{lem:MR-RM} and Lemma~\ref{lem:MS=0}. Namely, expand $(M+zR)^{n\choose 2}$ and move $R$'s towards the right such that in each step, we replace $\cdots RM_k\cdots$ by $\cdots M_kR\cdots-\cdots M_{k+1}\cdots$, keeping the total degree. In the end when no such moves are possible, either $R$ appears on the right side, resulting in a term equal to 0 (since $R\cdot1=0$), or $\prod_{k\geq1}M_k^{p_k}$ appears with $\sum kp_k={n\choose 2}$, which is also 0 except the single term $M^{n\choose2}$.
\end{proof}

Theorem~\ref{thm:generalized-chevalley} now follows easily.
\begin{proof}[Proof of Theorem~\ref{thm:generalized-chevalley}]
The weights in Theorem \ref{thm:generalized-chevalley} are given by:
\[
\wt(v \lessdot vt_{ij})=\alpha_i + \alpha_{i+1} + \cdots + \alpha_{j-1} + (b_{v \lessdot vt_{ij}}-d_{v \lessdot vt_{ij}})z.
\]
Recall that we have
\begin{align*}
M\cdot\S_w=&\sum_{w\lessdot wt_{ij}}(\alpha_i+\cdots+\alpha_{j-1})\S_{wt_{ij}}\\
R\cdot\S_w=&\sum_{w\lessdot u}(b_{w\lessdot u}-d_{w\lessdot u})\S_u
\end{align*}
so putting them together,
$$(M+zR)\cdot\S_w=\sum_{w\lessdot u}\wt(w\lessdot u)\S_u.$$ An iteration (or induction) immediately gives
$$(M+zR)^{\ell}\cdot\S_w=\sum_{w\leq u,\ \ell(w)=\ell(u)-\ell}m_{\wt}(w,u)\cdot\S_u.$$
Taking $w=e$ and $\ell={n\choose2}$ in the above setting, we see that $m_{\wt}(e,w_0)$ is the coefficient of $\S_{w_0}$ in $(M+zR)^{n\choose2}$, modulo $I$. By Lemma~\ref{lem:MZR=M}, such coefficient does not depend on $z$. When $z=0$, our result is given by Stembridge \cite{Stembridge} (see also Stanley \cite{Stanley-Bruhat}). 
\end{proof}

\section{Proof of Theorem \ref{thm:generalized-macdonald}}
\label{sec:generalized-macdonald}
We first note a simple fact about the specialization of $\SS_w$: since $\SS_w$ has total $x$-degree $\ell(w)$ and total $y$-degree ${n\choose 2}-\ell(w)$, we have 
\[
(\Delta\SS_w)(1,\ldots,1)=\left({n\choose 2}-\ell(w)\right)\SS_w(1,\ldots,1).
\]
We then have the following lemma.
\begin{lemma}\label{lem:gen-mac-notdepend}
Fix $w\in S_n$. Then $$\sum_{u:\ w\lessdot u}\SS_u(1,\ldots,1)(b_{w\lessdot u}-c_{w\lessdot u})=0.$$
\end{lemma}
\begin{proof}

By Proposition \ref{prop:original-macdonald}, we have
\begin{align*}
    \mf{S}_u(1,\ldots,1) &= \frac{1}{\ell(u)!} \sum_{s_{a_1}\cdots s_{a_{\ell(u)}} = u} a_1 \cdots a_{\ell(u)} \\
    &= \frac{1}{\ell(u)!} \sum_{s_{a_1}\cdots s_{a_{\ell(u)}} = u} a_{\ell(u)} \cdots a_1 \\
    &= \mf{S}_{u^{-1}}(1,\ldots,1).
\end{align*}
Thus $\mf{S}_u$ and $\mf{S}_{u^{-1}}$ have the same number of monomials, so we have $\SS_u(1,\ldots,1)=\SS_{u^{-1}}(1,\ldots,1)$. Therefore, as $\ell(u)=\ell(u^{-1})$, $(\Delta\SS_w)(1,\ldots,1)=(\Delta\SS_{w^{-1}})(1,\ldots,1)$. In addition, notice that $w\lessdot u$ if and only if $w^{-1}\lessdot u^{-1}$ and that $b_{w\lessdot u}=c_{w^{-1}\lessdot u^{-1}}$ by Proposition \ref{prop:symmetries}.

Apply Proposition~\ref{prop:padded} to $w$ and $w^{-1}$ separately. We have
\begin{align*}
\Delta\SS_{w^{-1}}=&\sum_{u:\ w^{-1}\lessdot u}(1+2b_{w^{-1}\lessdot u^{-1}})\SS_{u}\\
=&\sum_{u^{-1}:\ w^{-1}\lessdot u^{-1}}(1+2b_{w^{-1}\lessdot u^{-1}})\SS_{u^{-1}}\\
=&\sum_{u:\ w\lessdot u}(1+2c_{w\lessdot u})\SS_{u^{-1}},\\
\Delta\SS_w=&\sum_{u:\ w\lessdot u}(1+2b_{w\lessdot u})\SS_u.
\end{align*}
Now take the principal specialization and subtract these two equations. The left-hand side becomes zero as explained above. Recalling from above that $\SS_{u^{-1}}(1,\ldots,1)=\SS_u(1,\ldots,1)$, we obtain the desired equality.
\end{proof}
Now we are ready to prove Theorem~\ref{thm:generalized-macdonald}.
\begin{proof}[Proof of Theorem~\ref{thm:generalized-macdonald}]
We proceed by induction on ${n\choose 2}-\ell(w)$. The base case $w=w_0$ is trivial as both sides equal 1. Now fix $w$ and assume that the statement is true for all $u$ with $\ell(u)>\ell(w)$. The following calculation is straightforward:
\begin{align*}
m_{\wt}(w,w_0)=&\sum_{u:\ w\lessdot u}(1+b_{w\lessdot u}(2-z)+c_{w\lessdot u}z)m_{\wt}(u,w_0)\\
=&\sum_{u:\ w\lessdot u}(1+b_{w\lessdot u}(2-z)+c_{w\lessdot u}z)\left({n\choose 2}-\ell(u)\right)!\cdot\SS_u(1,\ldots,1)\\
=&\left({n\choose2}-\ell(w)-1\right)!\sum_{u:\ w\lessdot u}(1+2b_{w\lessdot u})\SS_u(1,\ldots,1)\\
&-\left({n\choose2}-\ell(w)-1\right)!z\sum_{u:\ w\lessdot u}(b_{w\lessdot u}-c_{w\lessdot u})\SS_u(1,\ldots,1).
\end{align*}
By Lemma~\ref{lem:gen-mac-notdepend}, the second term in the above expression becomes 0. And by the principal specialization of Proposition~\ref{prop:padded}, we have that 
\[
\sum_{u:\ w\lessdot u}(1+2b_{w\lessdot u})\SS_u(1,\ldots,1)=({n\choose 2}-\ell(w))\SS_w(1,\ldots,1).\] 
Thus the first term in the above expression becomes $({n\choose 2}-\ell(w))!\S_w(1,\ldots,1)$, which is what we want.
\end{proof}


We can now complete the proof of Theorem \ref{thm:any-pair}.

\begin{proof}[Proof of Theorem \ref{thm:any-pair}]
There are six cases to consider, depending on which pair of $z_A,z_B,z_C,$ and $z_D$ are equal to $z$ and $2-z$ (the others being zero); which element of the pair is sent to $z$ or $2-z$ does not matter, since the claimed result is independent of $z$.

For the pair $\{z_B, z_C\}$, letting $w=e$ in Theorem \ref{thm:generalized-macdonald} proves the result.  For $\{z_B,z_D\}$, letting $\alpha_1=\cdots =\alpha_{n-1}$ in Theorem \ref{thm:generalized-chevalley} gives weights
\begin{align*}
    \wt(v \lessdot w)=(1+b_{v \lessdot w}+d_{v \lessdot w})+(b_{v\lessdot w}-d_{v\lessdot w})z
\end{align*}
which clearly give all of the desired linear combinations of $b_{v\lessdot w}$ and $d_{v \lessdot w}$.

Applying the symmetries from Proposition \ref{prop:symmetries} then yields the remaining pairs.
\end{proof}

\section*{Acknowledgements}

The authors wish to thank Alex Postnikov and John Stembridge for helpful suggestions.

\bibliographystyle{plain}
\bibliography{main}
\end{document}